\documentclass[a4paper,12pt]{article}

\usepackage{amssymb}
\usepackage{tikz-cd}
\usetikzlibrary{cd}
\usetikzlibrary{babel}
\usepackage{amsmath}
\usepackage{comment}
\usepackage{cmap}					% поиск в PDF
\usepackage[T2A]{fontenc}			% кодировка
\usepackage[utf8]{inputenc}			% кодировка исходного текста
\usepackage[russian,english]{babel}	% локализация и переносы

\DeclareMathOperator{\Ima}{Im}

\newenvironment{proof} % имя окружения
{\par\noindent{\bf Proof.}} % команды для \begin
{\hfill$\scriptstyle\blacksquare$} % команды для \end

\newtheorem{theorem}{Theorem}[section]
\newtheorem{lemma}{Lemma}[section]
\newtheorem{proposal}{Proposal}[section]
\newtheorem{definition}{Definition}[section]

\makeatletter
\renewcommand\@date{{%
  \vspace{-\baselineskip}%
  \large\centering
  \begin{tabular}{@{}c@{}}
    Andronick Arutyunov\textsuperscript{1} \\
    \normalsize andronick.arutyunov@gmail.com
  \end{tabular}%
  \quad \quad
  \begin{tabular}{@{}c@{}}
    Alekseev Aleksandr\textsuperscript{2} \\
    \normalsize aleksandr.alekseev@frtk.ru
  \end{tabular}

  \bigskip

  \textsuperscript{1}Department of Higher Mathematics, Moscow institute of physics and technology\par
  \textsuperscript{2} Department of Control and Applied Mathematics, Moscow institute of physics and technology

  \bigskip

  \today
}}
\makeatother

\title{Cohomology of n-categories and derivations in group algebras \thanks{This work is supported by grant MK-1938.2017.1}}

\begin{document} % Конец преамбулы, начало текста.

\maketitle

\section{Abstract}
\quad There is a well-known problem of describing the algebra of the derivations in group algebra. In \cite{litlink4}, a method was proposed for describing derivations by considering the character space of a certain groupoid associated with a given group algebra. This groupoid is a action groupoid generated by the inner automorphisms of the group.

This construction allows some generalizations, which this work is devoted to.

It turns out that a given $ n- $category can be associated with a set of vector spaces, the so-called $ n- $characters, i.e. mappings from the set of $ n- $morphisms of the $ n- $category to complex numbers preserving the composition. We show that the corresponding sequence is exact, that is, in fact, we construct the cohomology of a given $ n- $category. In the particular case when this category is a $ 2- $groupoid associated with a group algebra, the complex gives a description of the derivation algebra in the group algebra.

Within the framework of this paper, we will also propose a natural modification of the construction of internal and external derivations, which significantly encapsulates the study of the space of derivations, taking into account the geometric constructions that we introduce.

\subsection{General definitions}

\quad An important role for us will be played by the notion of the $ n- $category and, in particular, the $ 2- $category (\cite{litlink1}, \cite{litlink3}, p. 312).

We start with the definition of the $ 2- $category.

\begin{definition} \textit {$ 2- $category} $C^{2}$ \textit{- is:}
\begin{itemize}
\item A class $Obj(C^2)$, whose elements are called objects
\item A class $\textbf{Hom}^1 (a, b)$ for every $a$ and $b$, whose elements are called maps, object $a$ is called a source, $b$ is called a target. Every map $\varphi$ can be written like an expression: $\varphi : a \Rightarrow b$.  $\textbf{Hom}^1 (C^2)$ -- a class of the all 1-maps.
\item A binary operation $\circ$ is called composition of maps $\varphi \in$ $\textbf{Hom}^1 (a, b)$ and $\psi \in$ $\textbf{Hom}^1 (b, c)$ and written as $\varphi \circ \psi$ or $\varphi \psi \in$ $\textbf{Hom}^1 (a, c)$, governed by two axioms:
\begin{enumerate}
\item Associativity: If $\varphi : a \Rightarrow b$, $\psi : b \Rightarrow c$ and $\xi : c \Rightarrow d$ then $\psi \circ (\xi \circ \varphi) = (\psi \circ \xi) \circ \varphi$, and
\item Identity: For every object $a$, there exists a map $1_a : a \Rightarrow a$ called the identity map for $a$, such that for every map $\varphi : a \Rightarrow b$, we have $1_b \circ \varphi = \varphi = \varphi \circ 1_a$.
\end{enumerate}
\item A class $\textbf{Hom}^2 (\varphi, \psi)$ for every  $\varphi, \psi \in$ $\textbf{Hom}^1 (a, b)$, whose elements are called 2-maps. Every 2-map $\alpha$ can be written like an expression: $\alpha : \varphi \Rightarrow \psi$ 

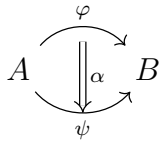
\begin{equation*}
\begin{tikzcd}
A \arrow[r, bend left=50, ""{name=U, below}, "\varphi"]
\arrow[r, bend right=50, ""{name=D}, swap, "\psi"]
& B
\arrow[Rightarrow, from=U, to=D, "\alpha"]
\end{tikzcd}
\end{equation*}

\item A binary operation $\bullet$ is called vertical composition of 2-maps $\alpha : \varphi \Rightarrow \psi$ and $\beta : \psi \Rightarrow \xi$ and written as $\alpha \bullet \beta : \varphi \Rightarrow \xi$, governed by two axioms:
\begin{enumerate}
\item Associativity: If $\alpha : \varphi \Rightarrow \psi$, $\beta : \psi \Rightarrow \xi$ and $\gamma : \xi \Rightarrow \eta$ then $\alpha \bullet (\beta \bullet \gamma) = (\alpha \bullet \beta) \bullet \gamma$, and
\item Identity: For every 1-map $\varphi$, there exists a 2-map $1_{\varphi} : \varphi \Rightarrow \varphi$ called the horizontal identity 2-map for $\varphi$, such that for every 2-map $\alpha : \varphi \Rightarrow \psi$, we have $1_{\psi} \bullet \alpha = \alpha = \alpha \bullet 1_{\varphi}$.
\end{enumerate}

\item A binary operation $\circ$ is called horizontal composition of 2-maps $\alpha : \varphi \Rightarrow \psi$ and $\alpha' : \varphi' \Rightarrow \psi'$ and written as $\alpha \circ \alpha'$ : $\varphi\varphi' \Rightarrow \psi\psi'$, governed by two axioms:
\begin{enumerate}
\item Associativity: If $\alpha : \varphi \Rightarrow \psi$, $\alpha' : \varphi' \Rightarrow \psi'$ and $\alpha'' : \varphi'' \Rightarrow \psi''$ then $\alpha \circ (\alpha' \circ \alpha'') = (\alpha \circ \alpha') \circ \alpha''$, and
\item Identity: For every object $a$, there exists a 2-map $1_{1_a} : 1_a \Rightarrow 1_a$ called the vertical identity 2-map for $a$, such that for every 2-map $\alpha : \varphi \Rightarrow \psi$, we have $1_{1_b} \circ \alpha = \alpha = \alpha \circ 1_{1_a}$.
\end{enumerate}

\begin{equation*}
\begin{tikzcd}
A \arrow[r, "\varphi", ""{name=U}] 
& B  \arrow[r, "\varphi'", ""{name=U1}]  & C
\\
A \arrow[r, swap, "\psi", ""{name=D}]
& B \arrow[r, swap, "\psi'", ""{name=D1}] & C
\arrow[Rightarrow, from=U, to=D, "\alpha"]
\arrow[Rightarrow, from=U1, to=D1, "\alpha'"] \\
A \arrow[r, "\varphi \circ \varphi'", ""{name=U2}] 
& C \\
A \arrow[r, swap, "\psi \circ \psi'", ""{name=D2}]
& C	
\arrow[Rightarrow, from=U2, to=D2, "\alpha \circ \alpha'"]
\end{tikzcd}
\end{equation*}

\item A middle four exchange equations for horizontal and vertical compositions
\begin{equation}
	(\alpha \circ \alpha') \bullet (\beta \circ \beta') = (\alpha \bullet \beta) \circ (\alpha' \bullet \beta').
\end{equation}
In the other words the represented diagram is commutative
\begin{equation*}
\begin{tikzcd}
A \arrow[r, ""{name=U}] 
& B  \arrow[r, ""{name=U1}]  & C
\\
A \arrow[r, ""{name=D}]
& B \arrow[r, ""{name=D1}] & C 
\\
A \arrow[r, ""{name=C}]
& B \arrow[r, ""{name=C1}] & C
\arrow[Rightarrow, from=U, to=D, "\alpha"]
\arrow[Rightarrow, from=U1, to=D1, "\alpha'"]
\arrow[Rightarrow, from=D, to=C, "\beta"]
\arrow[Rightarrow, from=D1, to=C1, "\beta'"]
\end{tikzcd}
\end{equation*}
\end{itemize}
\end{definition}

This definition can be expand to higher dimensions. Here we introduce n-category definition. (\cite{litlink1}, \cite{litlink5} p. 17). All the categories considered below are small.

\begin{definition}
\textit {n-category $C^n$ - is a set of the sets $\textbf{Hom}^0(C^n)$, $\textbf{Hom}^1(C^n)$, $\textbf{Hom}^2(C^n)$, $\dots$, $\textbf{Hom}^n(C^n)$, where:}
\begin{itemize}
\item Elements of the set $\textbf{Hom}^m(C^n)$ are called m-maps
\item For each $0 \leq k < m \leq n$ sets $\textbf{Hom}^k(C^n)$ and $\textbf{Hom}^m(C^n)$ creates a category, where $\textbf{Hom}^k(C^n)$ is an object class, $\textbf{Hom}^m(C^n)$ is a set of maps. Composition of the maps $\varphi$ and $\psi \in \textbf{Hom}^m(C^n)$ writes as $\varphi \circ_k \psi$.
\item For each $0 \leq h < k < m \leq n$ sets $\textbf{Hom}^h(C^n)$, $\textbf{Hom}^k(C^n)$ and $\textbf{Hom}^m(C^n)$ creates a 2-category.
\end{itemize}
\end{definition}

$ 0- $maps are usually called objects, and $ 1- $maps usually called maps.

\begin{definition}

\textit {An n-groupoid is an n-category whose all maps have inverse.}

\end{definition}

\begin{definition}

\textit {A totally-connected n-groupoid is an n-groupoid whose all possible pairs $\textbf{Hom}^k (C^n)$ and $\textbf{Hom}^m (C^n) $ form a totally-connected 1-groupoid, which is a groupoid that has at least one map between objects where a map can exist.}

\end{definition}

We consider an arbitrary m-morphism $\alpha : \varphi \Rightarrow \psi$. Define the maps $s^m, t^m : \textbf{Hom}^m(C^n) \rightarrow \textbf{Hom}^{m-1}(C^n)$ such, that
\begin{equation}
	s^m(\alpha) = \varphi
\end{equation}
\begin{equation}
	t^m(\alpha) = \psi
\end{equation}

\begin{equation*}
\begin{tikzcd}
\, \arrow[r, "s^m(\alpha)", ""{name=U2}] 
& \, \\
\, \arrow[r, swap, "t^m(\alpha)", ""{name=D2}]
& \,
\arrow[Rightarrow, from=U2, to=D2, "\alpha"]
\end{tikzcd}
\end{equation*}

\subsection{Characters}
\quad Now we define the concept of $ k- $character on the $ n- $category, counting $k \leq n $. As in the case of the definition of $ n- $categories, we will give the definition sequentially, beginning with small $ n $.

\begin{definition}

\textit{A map $\chi_0 : G \rightarrow \mathbb{C}$ is called 0-character on the object set $G$.}

\end{definition}

The concept of $ 1- $character has already been applied to the description of the algebra of derivations (\cite{litlink4} p. 17)

\begin{definition}
\textit{A map $\chi_1 : \textbf{Hom}^1 (C) \rightarrow \mathbb{C}$ , such that for every 1-maps $\varphi$ and $\psi$}
\begin{center}
	$\chi(\varphi \circ \psi) = \chi(\varphi) + \chi(\psi)$ 
\end{center}
is called 1-character on the 1-category
\end{definition}

As in the case of maps, $ 1- $characters are usually called simply characters.

\begin{definition}
\textit{1-character $\chi_1$, such that $\forall a, \forall \varphi \in $ $\textbf{Hom}^1 (a, a) \Rightarrow \chi_1 (\varphi) = 0$ is called trivial on loops}
\end{definition}

\begin{definition}
\textit{A map $\chi_2 : \textbf{Hom}^{2}(C^2) \rightarrow \mathbb{C}$ such that for every $\alpha, \alpha', \beta$ }
\begin{center}
	$\chi_2(\alpha \circ \alpha') = \chi_2(\alpha) + \chi_2(\alpha')$ 
    \\$\chi_2(\alpha \bullet \beta) = \chi_2(\alpha) + \chi_2(\beta)$ 
\end{center}
\begin{equation*}
\begin{tikzcd}
A \arrow[r, ""{name=U}] 
& B  \arrow[r, ""{name=U1}]  & C
\\
A \arrow[r, ""{name=D}]
& B \arrow[r, ""{name=D1}] & C 
\\
A \arrow[r, ""{name=C}]
& B \arrow[r, ""{name=C1}] & C
\arrow[Rightarrow, from=U, to=D, "\alpha"]
\arrow[Rightarrow, from=U1, to=D1, "\alpha'"]
\arrow[Rightarrow, from=D, to=C, "\beta"]
\arrow[Rightarrow, from=D1, to=C1, "\beta'"]
\end{tikzcd}
\end{equation*}
is called a 2-character on the 2-category
\end{definition}

\begin{definition}
\textit{A 2-groupoid $\Gamma^2$, such that}

\begin{enumerate}
\item \textit{$\Gamma^2$ - totally-connected}
\item \textit{A set $\textbf{Hom}^{2}(\varphi, \psi)$ contains exactly one 2-map if there can exist a 2-map between $ \varphi $ and $ \psi $.}
\end{enumerate}
is called primitive 2-groupoid
\end{definition}

 Let $X_k$ be a set of all k-characters on the $\Gamma^2$.

Let $\varphi_0 : \mathbb{C} \rightarrow X_0$ be a map such that
\begin{equation}
	[\varphi_0(z)](g) = z
\end{equation}
where $z \in \mathbb{C}, g \in G$

Let $\varphi_n : X_{n-1} \rightarrow X_n$ for $n = 1, 2$ be a map such that
\begin{equation}
	[\varphi_n(\chi)](\varphi) = \chi(t^n(\varphi)) - \chi(s^n(\varphi))
\end{equation}

In the following statements we show the correctness of this definition.

\begin{proposal}
\textit{A map $\varphi_1 : X_0 \rightarrow X_1$ is a homomorphism of the linear spaces and $\varphi_1(\chi)$ is trivial on loops.}
\end{proposal}

\begin{proof}
\begin {enumerate}
\item $\varphi_1(\chi_0)$ is a 1-character, since for 1-maps $\varphi \in \textbf{Hom}^1(a, b)$ и $\varphi' \in \textbf{Hom}^1(b, c)$
\begin{center}
	$\varphi_1(\chi_0)(\varphi \circ \varphi') = \chi_0(c) + \chi_0(b) - \chi_0(b) - \chi_0(a) = \varphi_1(\chi_0)(\varphi) + \varphi_1(\chi_0)(\varphi')$
\end{center}
\item $\varphi_1(\cdot)$ is a homomorphism of the linear spaces, since 
\begin{center}
$\varphi_1(\chi_0^{(1)} + \chi_0^{(2)}) = \varphi_1(\chi_0^{(1)}) + \varphi_1(\chi_0^{(2)})$
\end{center}
\item $\varphi_1(\chi_0)$ is trivial on loops, since for every $\varphi \in \textbf{Hom}^1(a, a)$ 
\begin{center}
	$\varphi_1(\chi_0)(\varphi) = \chi_0(a) - \chi_0(a) = 0$
\end{center}
\end{enumerate}
\end{proof}

\begin{proposal}
\textit{A map $\varphi_2 : X_1 \rightarrow X_2$ is a homomorphism of the linear spaces and $\varphi_2$ is an epimorphism.}
\end{proposal}

\begin{proof}
\begin {enumerate}
\item $\varphi_2(\chi_1)$ is a 2-character, since for $\varphi, \psi, \eta \in \textbf{Hom}^1(a, b)$, $\varphi', \psi' \in \textbf{Hom}^1(b, c)$ и $\alpha : \varphi \Rightarrow \psi$, $\alpha' : \varphi' \Rightarrow \psi'$, $\beta : \psi \Rightarrow \eta$ 
\begin{center}
	$\varphi_2(\chi_1)(\alpha \circ \alpha') = \chi_1(\psi') + \chi_1(\psi) - \chi_1(\varphi) - \chi_1(\varphi') = \varphi_2(\chi_1)(\alpha) + \varphi_2(\chi_1)(\alpha')$
    \\$\varphi_2(\chi_1)(\alpha \bullet \beta) = \chi_1(\eta) - \chi_1(\varphi) = \varphi_2(\chi_1)(\alpha) + \varphi_2(\chi_1)(\beta)$
\end{center}
\item $\varphi_2(\cdot)$ is a homomorphism of the linear spaces, since 
\begin{center}
$\varphi_2(\lambda \chi_1^{(1)} + \mu \chi_1^{(2)}) = \lambda\varphi_2(\chi_1^{(1)}) + \mu\varphi_2(\chi_1^{(2)})$
\end{center}
\item $\varphi_2(\cdot)$ is an epimorphism.

 $\forall \chi_2 \in X_2$ $\exists \chi_1 \in X_1$ : $\varphi_2(\chi_1) = \chi_2$. Let create $\chi_1$ on given $\chi_2$. Let $\chi_1(1_a) = 0$ $\forall a \in Obj(\Gamma^2)$. Fix the object $a$. Let $F$ - be a set, where only one representative from each $\textbf{Hom}^1 (a, b)$ (The choice of the set $ F $ is not unique). Let add in $F$ an inverse to every 1-map. Then add all of its compositions. As a result, $\forall a, b \in Obj(\Gamma^2)$ $\exists \varphi \in \textbf{Hom}^1(a, b)$ such that $\varphi \in F$, also $\varphi, \psi \in F \rightarrow \varphi \circ \psi \in F$, if composition exists. Let $\chi_1(\varphi) = 0$ $\forall \varphi \in F$. Then if $\psi \notin F$, то $\exists \varphi \in F, \alpha \in \textbf{Hom}^2(\varphi, \psi)$ such that $\alpha : \varphi \Rightarrow \psi$. Let $\chi_1(\psi) = \chi_2(\alpha)$. Adjusted $\chi_1$ is a 1-character on $\Gamma^2$, such $\chi_1(\varphi \circ \psi) = \chi_1(\varphi) + \chi_1(\psi)$, since
\begin{center}
	$\exists \varphi' \in F, \alpha \in \textbf{Hom}^2(\varphi, \varphi')$ such that $\alpha : \varphi' \Rightarrow \varphi$
    \\$\exists \psi' \in F, \beta \in \textbf{Hom}^2(\psi, \psi')$ such that $\beta : \psi' \Rightarrow \psi$
\end{center}
$\chi_2$ is a 2-character, then
\begin{center}
	$\chi_2(\alpha \circ \beta) = \chi_2(\alpha) + \chi_2(\beta)$
\end{center}
Then by construction of $\chi_1$:
\begin{center}
	$\chi_2(\alpha \circ \beta) = \chi_1(\varphi \circ \psi)$
    \\$\chi_2(\alpha) = \chi_1(\varphi)$, $\chi_2(\beta) = \chi_1(\psi)$
    \\$\chi_1(\varphi \circ \psi) = \chi_1(\varphi) + \chi_1(\psi)$
\end{center}

It follows that $ \varphi_2 $ is an epimorphism.
\end {enumerate}
\end{proof}

Let $\varphi_s : 0 \rightarrow \mathbb{C}$ be a map such that
\begin{equation}
	\varphi_s(0) = 0 \in \mathbb{C}
\end{equation}

Let $\varphi_t : X_2 \rightarrow 0$ be a map such that
\begin{equation}
	\varphi_t(\chi_2) = 0 \, \forall \chi_2 \in X_2
\end{equation}

Let prove the following important proposal, which will be proved later for an arbitrary totally-connected n-groupoid $ \Gamma^n $.

\begin{lemma}
\textit{A short sequence of algebraic objects and homomorphisms}
\begin{equation}
0 \xrightarrow[]{\varphi_s} \mathbb{C} \xrightarrow[]{\varphi_0} X_0 \xrightarrow[]{\varphi_1} X_1 \xrightarrow[]{\varphi_2} X_2 \xrightarrow[]{\varphi_t} 0
\end{equation}
\textit{is an exact short sequence.}
\end{lemma}

\begin{proof}
For this sequence to be exact, the following relations must be satisfied
\begin{center}
$\Ima\varphi_{i} = \ker\varphi_{i + 1}$, $i \in \{s, 0, 1, 2, t\}$.
\end{center}

Let's check it out.

\begin{enumerate}
\item $\Ima\varphi_t = \{0\}$. $\ker\varphi_0 = \{0\}$ as well, since $\varphi_0(z) = 0 \Leftrightarrow z = 0$ by design. A zero character is a character everywhere equal to zero.
\item $\Ima\varphi_0$ - is a set of the 0-characters such that $\chi_0 \in \Ima\varphi_0 \Leftrightarrow \chi_0(g) = \chi_0(h)$ for every $g, h \in G$. Since $[\varphi_1(\chi_0)](\varphi) = \chi_0(h) - \chi_0(g)$ for $\varphi \in \textbf{Hom} (g, h)$, then $[\varphi_1(\chi_0)](\varphi) = 0 \Leftrightarrow \chi_0 \in \Ima\varphi_0$.
\item $\Ima\varphi_1$ - set of the trivial on loops 1-characters. Let's show that $\varphi_2(\chi_1) = 0 \Leftrightarrow \chi_1$ - is trivial on loops.
\begin{itemize}
\item Let's first verify that if $\varphi_1(\chi_1) = 0 $, leads the fact that $\chi$ is trivial on loops. If $\varphi_1(\chi_1) = 0$, then $\forall \alpha : \varphi \Rightarrow \psi$ satisfied $\chi_1(\varphi) = \chi_1(\psi)$. Let's fix the objects $a$ и $b$. Let's consider $\xi \in$ $\textbf{Hom}^1 (a, a)$, such that $\xi = \varphi \circ \psi^{-1}$, $\psi^{-1}$ exists, since $\Gamma^2$ is a groupoid. Then $\chi_1(\xi) = \chi_1(\varphi) + \chi_1(\psi^{-1})$ ,but $\chi_1(\psi) = -\chi_1(\psi^{-1})$, which means that $\chi_1(\xi) = 0$

\item Inversely, $\varphi_1(\chi) = 0 \Leftarrow \chi$ -- is trivial on loops.
Let's fix the objects $a$ и $b$. Let's consider $\xi \in$ $\textbf{Hom}^1 (a, a)$, such that $\xi = \varphi \circ \psi^{-1}$ for $\psi, \varphi \in$ $\textbf{Hom}^1 (a, b)$. By condition $\chi_1(\xi) = 0$, which means $\chi_1(\varphi) = -\chi_1(\psi^{-1}) = \chi_1(\psi)$, i.e $\forall \alpha : \varphi \Rightarrow \psi$, $\varphi_1(\chi_1)(\alpha) = 0$, or $\varphi_1(\chi_1) = 0$.

\end{itemize}

\item $\ker \varphi_t = X_2$ by design, $\Ima \varphi_2 = X_2$, since $\varphi_2$ - epimorphism.

\end{enumerate}
\end{proof}

Our next goal is to generalize this lemma to the general case of $ n- $categories.

\begin{definition}
A map $\chi_m : \textbf{Hom}^m(C^n) \rightarrow \mathbb{C}, m \geq 1$, such that for every $\varphi, \psi \in \textbf{Hom}^m(C^n)$ and for every $0 \leq k < m \leq n$, where composition $\varphi \circ_k \psi$ is possible,
\begin{center}	
	$\chi_m(\varphi \circ_k \psi) = \chi_m(\varphi) + \chi_m(\psi)$
\end{center}
is called m-character on the n-category $C^n$.
\end{definition}

Let $X_k$ be a set of all k-characters on the groupoid $\Gamma^{\infty}$.

Let $\varphi_m : X_{m-1} \rightarrow X_m$ for $m \geq 1$ be a map such that
\begin{equation}
	[\varphi_n(\chi)](\varphi) = \chi(t^n(\varphi)) - \chi(s^n(\varphi))
\end{equation}

In the following statements we show the correctness of this definition.

\begin{proposal}
\textit{A map $\varphi_m$ is a homomorphism oh the linear spaces.}
\end{proposal}

\begin{proof}
\begin {enumerate}
\item Let fix $\textbf{Hom}^{k}(\Gamma^n)$, $0 \leq k < m \leq n$. Let $\varphi, \psi \in \textbf{Hom}^{m}(\Gamma^n)$ such that $\varphi : a \rightarrow b$, $\psi : b \rightarrow c$, $\varphi \circ_k \psi : a \rightarrow c$ for $a, b, c \in \textbf{Hom}^{k}(\Gamma^n)$. Let's proof that $\varphi_m(\cdot)$ is a m-character :
\begin{center}
	$\chi_m(\varphi \circ_k \psi) = \chi_{m-1}(c) - \chi_{m-1}(a)$
    \\$\chi_{m-1}(c) - \chi_{m-1}(a) = [\chi_{m-1}(c) - \chi_{m-1}(b)] + [\chi_{m-1}(b) - \chi_{m-1}(a)]$
    \\$\chi_m(\varphi \circ_k \psi) = \chi_m(\varphi) + \chi_m(\psi)$
\end{center}
\item $\varphi_m(\cdot)$ is a homomorphism, since for every $\chi_{m-1}^{(1)}, \chi_{m-1}^{(2)} \in X_{m-1}$ the following equation is performed
\begin{center}
$\varphi_m(\chi_{m-1}^{(1)} + \chi_{m-1}^{(2)}) = \varphi_m(\chi_{m-1}^{(1)}) + \varphi_m(\chi_{m-1}^{(2)})$
\end{center}
\end{enumerate}
\end{proof}

Let $\varphi_1 : \mathbb{C} \rightarrow X_0$ be a map, such that
\begin{equation}
	[\varphi_1(z)](g) = z
\end{equation}
where $z \in \mathbb{C}, g \in G$

\begin{theorem}
\textit{A sequence of algebraic objects and homomorphisms}
\begin{center}
$0 \xrightarrow[]{} \mathbb{C} \xrightarrow[]{\varphi_0} X_0 \xrightarrow[]{\varphi_1} X_1 \rightarrow \cdots \rightarrow X_{m-1} \xrightarrow[]{\varphi_m} X_m \rightarrow \cdots$
\end{center}
\textit{is an exact sequence.}
\end{theorem}

\begin{proof}
For this sequence to be exact, the following relations must be satisfied:
\begin{center}
$\Ima\varphi_{i} = \ker\varphi_{i + 1}$, $i \in \mathbb{N}_0$
\end{center}
Let's check it out:
\begin{enumerate}
\item $\Ima\varphi = \{0\}$. Also $\ker\varphi_0 = \{0\}$, since $\varphi_0(z) = 0 \Leftrightarrow z = 0$ by design. 
\item A category, where a set $\textbf{Hom}^{m-2}(\Gamma^n)$ is an object set, $\textbf{Hom}^{m-1}(\Gamma^n)$ -- is a set of 1-morphisms and $\textbf{Hom}^{m}(\Gamma^n)$ -- is a set of 2-morphisms, is a 2-category by definition. And it is a 2-groupoid in our case.
\begin{center}
$X_{m-2} \xrightarrow[]{\varphi_{m-1}} X_{m-1} \xrightarrow[]{\varphi_m} X_m$
\end{center}
We apply Lemma 1.1 to this section. This groupoid is not primitive, but in this case the epimorphicity of the morphism $\varphi_m$ is not important, so the lemma is applied correctly.

\end{enumerate}
\end{proof}

In terms of the constructed sequence, it is possible to describe a primitive 2-groupoid:

\begin{proposal}
\textit{If for some n-groupoid, the 2-groupoid formed by the sets $ \textbf {Hom}^{0} (\Gamma^n) $, $ \textbf{Hom}^{1} (\Gamma^n) $ and $ \textbf{Hom}^{2} (\Gamma^n) $ is primitive, then the cohomology ($ \ker \varphi_m \setminus \Ima \varphi_{m-1} $) starting with $ m = 3 $ are trivial.}
\end{proposal}

\section{The 2-category example}

Let's create a 2-groupoid $\Gamma^2$ based on infinite noncommutative group G in the following way (\cite{litlink4} p. 18):
\begin {itemize}
\item $Obj(\Gamma^2) = G$
\item \textbf{Hom} $(a, b) = \{(u, v)\in G \times G | v^{-1}u = a, uv^{-1} = b\}$ for every $a$, $b \in Obj(\Gamma^2)$
\item Composition of maps $\varphi = (u_1, v_1) \in$ \textbf{Hom} $(a, b)$, $\psi = (u_2, v_2) \in$ \textbf{Hom} $(b, c)$ is a map $\varphi \circ \psi \in$ \textbf{Hom} $(a, c)$ such that:
\begin{center}
	$\varphi \circ \psi = (u_2v_1, v_2v_1)$
\end{center}
\item A single 2-map $\alpha : \varphi \Rightarrow \psi$ is defined for every $\varphi, \psi \in$ \textbf{Hom} $(a, b)$

\end {itemize}

$\Gamma^2$ can be represented as a disjoint union (\cite{litlink4} p. 19):
\begin{center}
	$\Gamma^2 = \underset{[u] \in [G]}{\bigsqcup} \Gamma^2_{[u]}$
\end{center}

\section{The connection between the m-characters and derivations of the group algebra}
\subsection{General definitions}
Let $d$ be a differentiation operator. Consider element $u$ of the group algebra $\mathbb{C}[G]$, which can be represented as $u = \sum\limits_{g\in G}\lambda^g g$, where sum is finite. Then element $d(u)$ can be represented as (\cite{litlink4} p. 17)
\begin{equation}
\label{eqDER}
  d(u) = \sum\limits_{g\in G}\left(\sum\limits_{h\in G} d^h_g \lambda^h \right)g,
\end{equation}
where $d^g_h\in \mathbb{C}$ -- coefficients that depend only on the derivation $d$.

Let $\chi_d:\textbf{Hom}(\Gamma)\to \mathbb{C}$ be a map, based on the derivation $d$, defined as
\begin{equation}
\label{d^g_h}
  \chi_d ((h,g)) = d^h_g.
\end{equation}

\begin{proposal}
\textit{A map $\chi_d$ is a 1-character.}
\end{proposal}

\begin{definition}
\textit{A 1-character $\chi$ -- such that for every group element $v\in G$}
  $$
    \chi(x,v) = 0,
  $$
  \textit{almost for every $x\in G$ is called locally finite 1-character.}
\end{definition}

1-character $\chi$ sets the derivation if and only if it is locally finite ($[1]$).

Let $X_1^{fin}$ be a set of the all locally finite 1-characters on the $\Gamma^2$.

Let $X_0^{fin}$ be a set such that
\begin{center}
	$X_0^{fin} = \{\chi_0 \in X_0 \, | \, \varphi_2(\chi_0) \in X_1^{fin}\}$
\end{center}

Let $X_2^{fin}$ be a set such that
\begin{center}
	$X_2^{fin} = \varphi_3(X_1^{fin})$
\end{center}

Considered 2-groupoid $\Gamma^2$ is primitive, and the following short sequence is exact.
\begin{equation}
0 \xrightarrow[]{\varphi_s} \mathbb{C} \xrightarrow[]{\varphi_0} X_0^{fin} \xrightarrow[]{\varphi_1} X_1^{fin} \xrightarrow[]{\varphi_2} X_2^{fin} \xrightarrow[]{\varphi_t} 0
\end{equation}

\subsection{Examples of the derivations}

 A derivation is called internal if it is given by a formula (\cite{litlink4}, p. 21)
$$
  d_{a}: x\to [a, x], a\in \mathbb{C}[G].
$$

The Lie subalgebra of the inner derivations $Der_{inn} \subset Der$ is an ideal.

The quotient set $Der_{out}:= Der\setminus Der_{inn}$ is called a set of the outer derivations.

Consider an element $a \in G$. Let $\chi^a: \textbf{Hom} (\Gamma)\to \mathbb{C}$ be a map defined as follows. If $b\neq a$ -- is an element of the group $G$, then for every map $\phi\in \textbf{Hom} (a,b)$ with the source $a$ and the target $b$ let $\chi^a(\phi) = 1$, for every map $\psi\in \textbf{Hom} (b,a)$ let $\chi^a (\psi) = -1$. For the rest maps let $\chi^a$ equals to zero.

\begin{proposal}
\textit{The character $\chi^a$ is the 1-character. The 1-character $\chi^a$ is the 1-character defined by an inner differentiation:}
 \begin{equation}
 	d_a: x\to [x,a]
 \end{equation}
\end{proposal}

\subsection{Algebra of the 2-characters}

The linear space $Der$ is a Lie algebra with the commutator (\cite{litlink6}, p. 206):
\begin{center}
$[d_1, d_2] = d_1d_2 - d_2d_1$
\end{center}

\begin{proposal} \textit{The values of the 1-character $\chi_{[d_1, d_2]} = \{\chi_{d_1}, \chi_{d_2}\}$ are defined by $\chi_{d_1}$ and $\chi_{d_2}$ as follows}
\begin{equation}
	\{\chi_{d_1}, \chi_{d_2}\}(a, g) = \sum\limits_{h\in G} \chi_{d_1}(a, h)\chi_{d_2}(h, g) - \chi_{d_2}(a, h)\chi_{d_1}(h, g)
\end{equation}
\end{proposal}

\begin{proof}
	Let $g \in G$, then:
    \begin{center}
    	$d_1(g) = \sum\limits_{h\in G} \chi_{d_1}(h, g)h$
        \\$d_2(g) = \sum\limits_{h\in G} \chi_{d_2}(h, g)h$
        \\$[d_1, d_2](g) = \sum\limits_{h\in G} \{\chi_{d_1}, \chi_{d_2}\}(a, g)a$
    \end{center}
    
    Represent the expression for the commutator by definition:
    \begin{center}
    	$[d_1, d_2] = d_1d_2 - d_2d_1$
        \\$d_1d_2(g) = \sum\limits_{h\in G} \chi_{d_2}(h, g)(\sum\limits_{a\in G} \chi_{d_1}(a, h)a)$
        \\$d_2d_1(g) = \sum\limits_{h\in G} \chi_{d_1}(h, g)(\sum\limits_{a\in G} \chi_{d_2}(a, h)a)$
    \end{center}
    
    Change the sum order in the last expressions:
    
    \begin{center}
        $[d_1, d_2](h) = \sum\limits_{a\in G} (\sum\limits_{h\in G} \chi_{d_2}(h, g)\chi_{d_1}(a, h) - \chi_{d_1}(h, g)\chi_{d_2}(a, h))a$
    \end{center}
    
    The formula for $\{\chi_{d_1}, \chi_{d_2}\}(a, h)$ is a coefficient of $a$, i.e.
    
    \begin{center}
    	$\{\chi_{d_1}, \chi_{d_2}\}(a, g) = \sum\limits_{h\in G} \chi_{d_1}(a, h)\chi_{d_2}(h, g) - \chi_{d_2}(a, h)\chi_{d_1}(h, g)$
    \end{center}
    
\end{proof}

Define a weak inner derivation space $D^*_{Inn}$ as follows:
\begin{center}
	$D^*_{Inn} = \{d \in Der \, | \, \chi_d$ - trivial on loops$\}$
\end{center}

\begin{theorem}
\textit{Linear space $Der^*_{Inn} \subset Der$ is an ideal:}
\begin{center}
	$d_0 \in Der^*_{Inn}, d \in Der \Rightarrow [d_0, d], [d, d_0] \in Der^*_{Inn}$
\end{center}
\end{theorem}

\begin{proof}

Let's proof that $Der^*_{Inn} \subset Der$ is a subalgebra, i.e.
\begin{center}
	$d_1, d_2 \in D^*_{Inn} \Rightarrow [d_1, d_2] \in D^*_{Inn}$
\end{center}

	1-character $\chi_{d_1}$ can be represented as:
    \begin{center}
    	$\chi_{d_1} = \sum\limits_{a \in G} \lambda^a\chi^a$, $\lambda^a \in \mathbb{C}$
    \end{center}, 
since $\chi_{d_1}$ is trivial on loops, where $\chi^a$ is defined by formula $(14)$

Similarly with a 1-character $\chi_{d_2}$:
\begin{center}
    	$\chi_{d_2} = \sum\limits_{b \in G} \mu^b\chi^b$, $\mu^b \in \mathbb{C}$
\end{center}

Because of the bilinearity of the commutator,
\begin{center}
	$\{\chi_{d_1}, \chi_{d_2}\} = \sum\limits_{a \in G}\sum\limits_{b \in G}\lambda^a\mu^b \{\chi^a, \chi^b\}$
\end{center}

The commutator of the $d^a$ и $d^b$ can be represented as follows:
\begin{center}
	$[d^a, d^b] = d^{ab} - d^{ba}$
\end{center}

Then define the 1-character $\{\chi^a, \chi^b\}$ by the following formula
\begin{center}
    $\{\chi^a, \chi^b\} = \chi^{ab} - \chi^{ba}$
\end{center}

Obtain the final expression for $\{\chi_{d_1}, \chi_{d_2}\}$:
\begin{center}
    $\{\chi_{d_1}, \chi_{d_2}\} = \sum\limits_{a \in G}\sum\limits_{b \in G}\lambda^a\mu^b \chi^{ab} - \sum\limits_{a \in G}\sum\limits_{b \in G}\lambda^a\mu^b \chi^{ba}$
\end{center}

$\{\chi_{d_1}, \chi_{d_2}\} \in Der^*_{Inn}$, to proof that, concider the value of the character on the loop $(uz, z), z \in Z_G(u)$
\begin{center}
	$\{\chi_{d_1}, \chi_{d_2}\}(uz, z) = \sum\limits_{ab = zuz^{-1}}\lambda^a\mu^b - \sum\limits_{ab = u}\lambda^a\mu^b - \sum\limits_{ba = u}\lambda^a\mu^b + \sum\limits_{ba = zuz^{-1}}\lambda^a\mu^b = 0$
\end{center}

Now pass directly to the proof of the theorem. Represent $\chi_{d_0}$ as
\begin{center}
    $\chi_{d_0} = \sum\limits_{a \in G} \lambda^a\chi^a$
\end{center}
 
	Proof theorem for $\chi^a$ and extend the result to the $\chi_{d_0}$, using the bilinearity of the commutator. Consider 1-character $\{\chi_d, \chi^a\}$. Proof that it is trivial on loops, i.e. $\forall b \in G $ and for $\forall z \in Z_G(b)$ is performed $\{\chi_d, \chi^a\}(bz, z) = 0$. By the formula $(15)$:
    \begin{center}
    	$\{\chi_d, \chi^a\}(bz, z) = \sum\limits_{h\in G} \chi_d(bz, h)\chi^a(h, z) - \chi^a(bz, h)\chi_d(h, z)$
    \end{center}

$\chi^a(h, z) \neq 0$ only in 2 cases: when $h = za$ and $h = az$. $\chi^a(bz, h) \neq 0$ only in 2 cases: when $h = bza^{-1}$ и $h = a^{-1}bz$. It means that

	\begin{center}
    	$\{\chi_d, \chi^a\}(bz, z) = \chi_d(bz, za) - \chi_d(bz, az) + \chi_d(a^{-1}bz, z) - \chi_d(bza^{-1}, z)$
    \end{center}

But also
\begin{center}
	$(a^{-1}bz, z) \circ (bz, za) = (bz, az) \circ (bza^{-1}, z)$
\end{center}
That means
\begin{center}
	$\chi_d(bz, za) + \chi_d(a^{-1}bz, z) = \chi_d(bz, az) + \chi_d(bza^{-1}, z)$
    \\$\{\chi_d, \chi^a\}(bz, z) = 0$
\end{center}

Q.E.D.

\end{proof}

In accordance with the last statement, the space of weak external derivations becomes meaningful.
\begin{center}
	$Der^*_{Out}:= Der\setminus Der^*_{Inn}$
\end{center}

Similarly, the Lie subalgebra $X^{fin}_0 \subset X^{fin}_1$ is an ideal. Thus redefine space of the 2-characters as follows:
\begin{center}
	$X^{fin}_2 = X^{fin}_1 \setminus X^{fin}_0$
\end{center}

\begin{theorem}
\textit{Lie algebra $(X^{fin}_2, \{\cdot,\:\cdot\})$ is isomorphic to the Lie algebra $(Der^*_{Out}, [\cdot,\:\cdot]))$.  } 
\end{theorem}

\begin{proof}

Let $F$ be an isomorphism between $Der$ and $X^{fin}_1$, which is defined if paragraph 3.1. Thus define an isomorphism $F^*$ as follows 
\begin{center}
	$F^* : Der^*_{Out} \rightarrow X^{fin}_2$
    \\$F^*(d + Der^*_{Inn}) = F(d) + X^{fin}_0$
\end{center}

Show that it preserves the commutation operation. By the formula (15):
\begin{center}
	$\{F^*(d_1 + Der^*_{Inn}), F^*(d_2 + Der^*_{Inn})\} = \{F(d_1), F(d_2)\} + X^{fin}_0 = $
    \\$ = F^*([d_1, d_2] + Der^*_{Inn})$
\end{center}

\end{proof}

\end{document}